\documentclass{amsart}  

\usepackage{amsrefs}
\usepackage[usenames, dvipsnames]{color}
\usepackage{amsfonts}
\usepackage{amsthm}
\usepackage{amsmath}
\usepackage{amsfonts}
\usepackage{latexsym}
\usepackage{amssymb}
\usepackage{amscd}
\usepackage[latin1]{inputenc}
\usepackage{verbatim}
\usepackage{enumerate}
\usepackage{comment}
\usepackage{mathrsfs}

\newtheorem{theorem}{Theorem}[section]
\newtheorem{lemma}[theorem]{Lemma}
\newtheorem{proposition}[theorem]{Proposition}
\newtheorem{corollary}[theorem]{Corollary}
\newtheorem{example}[theorem]{Example}
\newtheorem{definition}[theorem]{Definition}

\theoremstyle{definition}

\newcommand\tr{ \mbox{Tr} }
\newcommand\leac{\ll_{\rm ac}}

\newcommand\ce{\simeq_{\rm cl} }

\begin{document}

\title{On Operator Valued Measures }

\author[Darian McLaren]{Darian McLaren\textsuperscript{1}}

\author[Sarah Plosker]{Sarah Plosker\textsuperscript{1,2,3}}

\author[Christopher Ramsey]{Christopher Ramsey\textsuperscript{1,2}}

\thanks{\textsuperscript{1}Department of Mathematics and Computer Science, Brandon University,
Brandon, MB R7A 6A9, Canada}
\thanks{\textsuperscript{2}Department of Mathematics, University of Manitoba, Winnipeg, MB  R3T 2N2, Canada}
\thanks{\textsuperscript{3}ploskers@brandonu.ca}


\maketitle

\begin{abstract} We consider positive operator valued measures whose image is the bounded operators acting on an infinite-dimensional Hilbert space, and we relax,  when possible,
    the usual assumption of positivity of the operator valued measure seen in the quantum information theory literature. We  define the Radon-Nikod\'ym derivative of a positive operator valued measure with respect to a complex measure induced by a given quantum state; this derivative does not always exist when the Hilbert space is infinite dimensional in so much as its range may include unbounded operators. We define integrability of a positive quantum random variable with respect to a positive operator valued measure. Emphasis is put on the structure of   operator valued measures, and  we develop positive operator valued versions of 
   the Lebesque decomposition theorem and Johnson's   atomic and nonatomic decomposition theorem. 
   Beyond these generalizations, we make connections between absolute continuity and the ``cleanness'' relation defined on positive operator valued measures as well as to the notion of atomic and nonatomic measures. 
    
\end{abstract}

\section{Introduction}

In (classical) measure theory,  $X$ is a set and $\Sigma$ is a $\sigma$-algebra over $X$; the pair $(X,\Sigma)$ then forms a measurable space. Classical measure theory includes well-known decomposition theorems such as the Hahn decomposition and the Jordan decomposition (often jointly referred to as the Hahn-Jordan decomposition), the Lebesgue decomposition, and Johnson's decomposition of a measure into atomic and nonatomic parts. 
While an operator valued analogue of the the Hahn-Jordan decomposition exists in the literature (see Section \ref{sec:HJ} for details), operator valued analogues of the remaining aforementioned classical results have not previously been considered.

Much work has been done recently to build up the mathematical foundations of a positive operator valued measure (POVM) (see \cites{clean2005, chiribella--etal2007, chiribella--etal2010, dariano--etal2005,   FFP, farenick--kozdron2012, FPS, jencova--pulmannov2009, heinonen2005, kahn2007, parthasarathy1999}, among others, as well as \cite{DL, Ozawa} for more classical treatments).  Depending on the mathematical analysis, the underlying set $X$ or Hilbert space $\mathcal{H}$ (or both) may be infinite or finite dimensional. Here, we keep our analysis fairly general by considering both $X$ and $\mathcal{H}$ to be infinite dimensional, although we assume $X$ is locally compact and Hausdorff as these assumptions afford us some structure to work with. Furthermore, whenever possible, we drop the assumption of positivity which appears in the quantum information theory literature.

Throughout,  $X$ is a locally compact Hausdorff space and  $\mathcal{O}(X)$   is the $\sigma$-algebra of Borel sets of $X$. For the following, recall that the predual of $\mathcal{B}(\mathcal H)$ is the ideal of trace class operators $\mathcal T(\mathcal H) = \mathcal{B}(\mathcal H)_*$. Of course, if $\mathcal H$ is finite-dimensional then $\mathcal T(\mathcal H) = \mathcal{B}(\mathcal H)$. We denote by $S(\mathcal H)\subset \mathcal T(\mathcal H)$ the convex set of positive operators of unit trace (such operators are called \emph{states} or \emph{density operators} and typically denoted by $\rho$). 

Following \cite{Larson et al, Paulsen} we have the following definition.

\begin{definition}
A map $\nu : \mathcal{O}(X) \to \mathcal{B}(\mathcal{H})$ is an \emph{operator-valued measure (OVM)} if it is weakly countably additive, meaning that for every countable collection $\{E_k\}_{k \in \mathbb N} \subseteq \mathcal{O}(X)$ with $E_j \cap E_k = \emptyset$ for $j \neq k$ we have
\[
\nu\left(\bigcup_{k\in \mathbb N} E_k \right) = \sum_{k \in \mathbb N} \nu(E_k)\,,
\]
where the convergence on the right side of the equation above is with respect to the ultraweak topology of $\mathcal{B}(\mathcal{H})$. We say $\nu$ is 
\begin{enumerate}[(i)]
    \item \emph{bounded} if $\sup\{\|\nu(E)\| : E\in \mathcal O(X)\} < \infty$,
    \item \emph{self-adjoint} if $\nu(E)^* = \nu(E)$, for all $E\in\mathcal O(X)$,
    \item \emph{positive} if $\nu(E) \in \mathcal{B}(\mathcal H)_+$, for all $E\in \mathcal O(X)$,
    \item \emph{spectral} if $\nu(E_1 \cap E_2) = \nu(E_1)\nu(E_2)$, for all $E_1,E_2\in \mathcal O(X)$
    \item \emph{regular} if the induced complex measure $\tr(\rho\nu(\cdot))$ is regular for every $\rho\in \mathcal T(\mathcal H)$.
\end{enumerate}
Moreover, $\nu$ is called a \emph{positive operator-valued probability measure} or \emph{quantum probability measure} if it is positive and $\nu(X) = I_{\mathcal{H}}$, and is called a \emph{projection-valued measure (PVM)} if it is self-adjoint and spectral.
\end{definition}

Note that it is automatic that a positive operator-valued measure (POVM) is bounded.

POVMs and quantum probability measures arose due to their interest in quantum information theory; see  \cite{Busch--Lahti--Mittelstaedt-book,Davies-book,Holevo-book2} as general references on this topic. OVMs in various forms have also been studied for quite a while but usually in quite a general way, often in conjunction with vector-valued measures \cite{Larson et al, Roth, VOVM book}. In many of these sources the convergence for the countable additivity of the OVM was assumed to be under the strong or weak operator topology, though if one is working with a bounded OVM then the ultraweak and weak topologies correspond.

A measurable set is \emph{$\sigma$-finite} if it can be expressed as the countable union of measurable sets with finite measure. We say that the measure $\mu$ or OVM $\nu$  is $\sigma$-finite if every measurable set is $\sigma$-finite; equivalently, a measure is $\sigma$-finite if $X$ is $\sigma$-finite. Note here again, that a POVM is bounded and thus is finite, not just $\sigma$-finite.

For $i=1,2$, let  $\vartheta_i:\mathcal{O}(X)\rightarrow \mathcal A_i$, where $\mathcal A_i$ is either $\mathcal{B}(\mathcal{H}_i)$  or the extended real number line, we say that $\vartheta_2$ is \emph{absolutely continuous} with respect to $\vartheta_1$ (denoted
$\vartheta_2 \ll_{\rm ac} \vartheta_1$) if $\vartheta_2(E)=0$ for all $E\in\mathcal{O}(X)$ with $\vartheta_1(E)=0$ (where $0$ is either the scalar zero or the zero operator, as appropriate). Note that a  measure $\mu_2$ can be absolutely continuous with respect to another  measure $\mu_1$ or with respect to an OVM $\nu$, and similarly an OVM $\nu_2$ can be absolutely continuous with respect to another OVM $\nu_1$ or with respect to a  measure $\mu$. 

A signed (classical) measure is an extended real-valued function on $\Sigma$ that is countably additive, assumes only one of the values $-\infty$ or $\infty$, and maps $\emptyset$ to 0.

In Section \ref{sec:integrals} we extend some recent results on POVMs on finite-dimensional Hilbert spaces to infinite-dimensional Hilbert spaces and relax the positivity assumption when possible. We define and develop a stronger variant of the integrability of a quantum random variable with respect to a POVM $\nu$ than what is found in the literature. In Section \ref{sec:decomp} we consider the structure of operator valued measures and  prove operator valued analogues of several well-known classical measure theory decomposition theorems. In Section \ref{sec:cleanIC} we prove some results related to the notion of informationaly complete quantum probability measures, and link the notion of a measurement basis and the partial order of cleanness with the property of atomic/non-atomic.  These are entirely quantum results, in that we do not see the concepts of informationally complete, measurement basis, and cleanness in classical measure theory.  

\section{Integrals of quantum random variables}\label{sec:integrals}

The main goal of this section is to extend the results of \cite{FPS} to the case of $\mathcal H$ being infinite-dimensional. In particular, $\mathcal H$ will always be separable.

A {\em quantum random variable} $f: X \rightarrow \mathcal{B}(\mathcal H)$ is a Borel measurable function between the $\sigma$-algebras generated by the open sets of $X$ and $\mathcal{B}(\mathcal H)$, respectively. 

Equivalently, $f$ is a quantum random variable if and only if 
\[
x\to \tr(\rho f(x))
\]
are Borel measurable functions for every state $\rho \in S(\mathcal H)$.

A quantum random variable $f: X \rightarrow \mathcal{B}(\mathcal H)$ is said to be 
\begin{enumerate}[(i)]
    \item \emph{bounded} if $\sup\{\|f(x)\| : x\in X\} < \infty$,
        \item \emph{normal} if $f(x)f(x)^* = f(x)^*f(x)$, for all $x\in X$,
    \item \emph{self-adjoint} if $f(x) = f(x)^*$, for all $x\in X$,
    \item \emph{positive} if $f(x) \in \mathcal{B}(\mathcal H)_+$, for all $x\in X$.
\end{enumerate}

A self-adjoint (or positive) quantum random variable can really be thought of as being Borel measurable from $X$ to $\mathcal{B}(\mathcal H)_{sa}$ (or $\mathcal{B}(\mathcal H)_+)$, where the range $\sigma$-algebra is still generated from the norm topology. This allows us to compose a quantum random variable with a continuous function to get another quantum random variable.

\begin{lemma}\label{lem:borel}
Let $f: X \rightarrow \mathcal{B}(\mathcal H)$ be a self-adjoint quantum random variable. Then $f_+, f_- : X \rightarrow \mathcal{B}(\mathcal H)_+$ defined by 
\[
f_+(x) = f(x)_+ \ \ \textrm{and} \ \ f_-(x) = f(x)_-, \ \ x\in X
\]
are positive quantum random variables. Similarly, if $f$ is a positive quantum random variable then $f^{1/2} : X \rightarrow \mathcal{B}(\mathcal H)_+$ defined by
\[
f^{1/2}(x) = f(x)^{1/2}
\]
is a positive quantum random variable.
\end{lemma}
\begin{proof}
It is a standard fact from functional calculus that $\max\{z,0\}, \min\{z,0\}$, and $z^{1/2}$ are continuous functions on $\mathcal{B}(\mathcal H)_{sa}$ and $\mathcal{B}(\mathcal H)_+$, respectively. Thus, they are Borel measurable and so $f_+ = \max\{z,0\}\circ f, f_- = -\min\{z,0\}\circ f$ and $f^{1/2} = z^{1/2}\circ f$ are Borel measurable functions.
\end{proof}

\begin{corollary}\label{cor:four}
Every quantum random variable is the linear combination of four positive quantum random variables.
\end{corollary}

Let $\nu: \mathcal O(X) \rightarrow \mathcal{B}(\mathcal H)$ be an OVM. For every state $\rho\in S(\mathcal H)$, the induced complex measure $\nu_\rho$ on $X$ is defined by
\[
\nu_\rho(E) = \tr(\rho \nu(E)), \ E\in \mathcal O(X).
\]
Note that because $\nu$ is weakly countably additive it follows that $\nu_\rho$ is countably additive.

If $\nu$ is a POVM and $\rho$ is a full-rank density operator then $\tr(\rho \,\cdot\,)$ maps nonzero positive operators to strictly positive numbers. Therefore, $\nu \ll_{\rm ac} \nu_\rho$ and $\nu_\rho \ll_{\rm ac} \nu$; that is, $\nu$ and $\nu_\rho$ are  mutually absolutely continuous for any full-rank $\rho\in S(\mathcal H)$. 

Assume $\mathcal H$ is separable and let $\{e_n\}$ be an orthonormal basis. Denote $\nu_{ij}$ the complex measure $\nu_{ij}(E) = \langle \nu(E)e_j,e_i\rangle, E\in \mathcal O(X)$. Now, for any full-rank density operator $\rho$ we have $\nu_{ij} \ll_{\rm ac} \nu_\rho$. Thus, by the classical Radon-Nikod\'ym theorem, there is a unique $\frac{d\nu_{ij}}{d\nu_\rho} \in L_1(X, \nu_\rho)$ such that
\[
\nu_{ij}(E) = \int_E \frac{d\nu_{ij}}{d\nu_\rho} d\nu_\rho, \ E\in \mathcal O(X).
\]

\begin{definition} Let $\nu:\mathcal{O}(X)\rightarrow \mathcal B(\mathcal H)$ be a POVM and $\rho\in S(\mathcal{H})$ be a full-rank density operator. 
The {\em Radon-Nikod\'ym derivative} of $\nu$ with respect to $\nu_\rho$ at the point $x\in X$ is defined implicitly by
\[
\left\langle \frac{d\nu}{d\nu_\rho}(x) e_j, e_i\right\rangle =  \frac{d\nu_{ij}}{d\nu_\rho}(x)
\]
which only is worth studying for us if it exists as a quantum random variable, meaning it takes every $x$ to a bounded operator.
\end{definition}

If $\nu$ is into a finite-dimensional Hilbert space then $\frac{d\nu}{d\nu_\rho}$ always exists. However, for infinite dimensions this derivative may not always exist. 
By Corollary \ref{cor:exists}, if $\frac{d\nu}{d\nu_\rho}$ exists for some full-rank density operator $\rho\in S(\mathcal H)$, then it exists for all full-rank density operators in $S(\mathcal H)$. Thus, we will   not specify a particular full-rank $\rho$ unless it is necessary to do so. 

\begin{example}  
Let $\mu$ be the Lebesgue measure on $[0,1]$ and define
\[
\nu(E) = {\rm diag}(\mu(E\cap[1/2, 1]), \mu(E\cap[1/3, 1/2]), \mu(E\cap[1/4,1/3]), \dots)
\]
which gives that $\nu : \mathcal O(X) \rightarrow \mathcal B(\mathcal H)$ is a POVM. Let $\rho = {\rm diag}(1/2, 1/4, \dots)$ be a full-rank density operator. Thus
\[
\nu_\rho(E) = \sum_{n\geq 1} \frac{1}{2^n}\mu(E \cap [1/(n+1), 1/n])
\]
and 
\[
\frac{d\nu}{d\nu_\rho} = {\rm diag}(2\chi_{[1/2,1]}, 4\chi_{[1/3,1/2]}, 8\chi_{[1/4,1/3]}, \dots)
\]
which is clearly not a quantum random variable. Therefore, there exists $\nu$ and $\rho$ such that $\frac{d\nu}{d\nu_\rho}$ does not exist.
\end{example}

\begin{proposition}
Let $\nu$ be a POVM such that  $\frac{d\nu}{d\nu_\rho}$ exists.  Then  $\frac{d\nu}{d\nu_\rho}$ 
is positive almost everywhere with respect to $\nu_\rho$.
\end{proposition}
\begin{proof}
For $n\geq 1$ and $\xi = \sum_{i=1}^n \xi_i e_i$, $\|\xi\| = 1$ we have that for every $E\in \mathcal O(X)$
\begin{align*}
\int_E \left\langle \frac{d\nu}{d\nu_\rho} \xi,\xi\right\rangle d\nu_\rho 
&= \int_E \left(\sum_{i,j=1}^n \frac{d\nu_{ij}}{d\nu_\rho} \xi_j \overline{\xi_i}\right) d\nu_\rho
\\ & = \sum_{i,j=1}^n \nu_{ij}(E)\xi_j\overline{\xi_i}
\\ & = \langle \nu(E)\xi,\xi\rangle \geq 0.
\end{align*}
From this one can see that on the subspace ${\rm span}\{e_1,\dots, e_n\}$ we have that $\frac{d\nu}{d\nu_\rho}(x) \geq 0$ almost everywhere with respect to $\nu_\rho.$ 
Therefore, this is true on all of $\mathcal H$ since the countable union of measure zero sets is measure zero.
\end{proof}

Following \cite{FPS} we define integrability with respect to a positive operator-valued measure.
\begin{definition}
Let $\nu : \mathcal O(X) \rightarrow \mathcal B(\mathcal H)$ be a POVM such that $\frac{d\nu}{d\nu_\rho}$ exists.
A positive quantum random variable $f: X \rightarrow \mathcal{B}(\mathcal H)$ is said to be {\em $\nu$-integrable} if the function
\[
f_s(x) = \tr\left( s\left(\frac{d\nu}{d\nu_\rho}(x) \right)^{1/2} f(x) \left(\frac{d\nu}{d\nu_\rho}(x) \right)^{1/2} \right), \ x\in X
\]
is $\nu_\rho$-integrable for every state $s\in S(\mathcal H)$. 

We say that an arbitrary quantum random variable $f: X \rightarrow \mathcal{B}(\mathcal H)$ is $\nu$-integrable if and only if $({\rm Re} f)_+, ({\rm Re} f)_-, ({\rm Im} f)_+$ and $({\rm Im} f)_-$ are $\nu$-integrable.
\end{definition}

Some classical results transfer over very well.

\begin{proposition}\label{prop:fgnu-int}
Let $f, g: X \rightarrow B(\mathcal H)$ be positive quantum random variables and $\nu : \mathcal O(X) \rightarrow B(\mathcal H)$ be a POVM such that $\frac{d\nu}{d\nu_\rho}$ exists. If $f(x) \leq g(x)$ almost everywhere with respect to $\nu$ and $g$ is $\nu$-integrable then $f$ is also $\nu$-integrable.
\end{proposition}
\begin{proof}
For any state $s\in S(\mathcal H)$ we have that 
\begin{align*}
f_s &= \tr\left(s\left(\frac{d\nu}{d\nu_\rho}\right)^{1/2}f\left(\frac{d\nu}{d\nu_\rho}\right)^{1/2}\right)
\\ &= \tr\left(s^{1/2}\left(\frac{d\nu}{d\nu_\rho}\right)^{1/2}f\left(\frac{d\nu}{d\nu_\rho}\right)^{1/2}s^{1/2}\right)
\\ &\leq \tr\left(s^{1/2}\left(\frac{d\nu}{d\nu_\rho}\right)^{1/2}g\left(\frac{d\nu}{d\nu_\rho}\right)^{1/2}s^{1/2}\right)
\\ &=g_s,
\end{align*}
almost everywhere with respect to $\nu_\rho$. Therefore, because $g_s$ is $\nu_\rho$-integrable then so is $f_s$ and thus $f$ is $\nu$-integrable.
\end{proof}

\begin{proposition}\label{prop:rdunique}
Let $\nu : \mathcal{O}(X) \rightarrow \mathcal B(\mathcal H)$ be a POVM such that $\frac{d\nu}{d\nu_\rho}$ exists.
The constant function $x\mapsto I_\mathcal H$ is $\nu$-integrable.
Moreover, $\frac{d\nu}{d\nu_\rho}$ is independent of choice of orthonormal basis almost everywhere with respect to $\nu_\rho$.
\end{proposition}
\begin{proof}
For any finite-rank operator $s = \sum_{(i,j)\in F} \lambda_{ij}e_{ij}$, for some finite set of indices $F$, we have
\begin{align*}
\int_X \tr\left(s\frac{d\nu}{d\nu_\rho}\right) d\nu_\rho & = \sum_{(i,j)\in F} \lambda_{ij} \int_X \frac{d\nu_{ji}}{d\nu_\rho} d\nu_\rho
\\ & = \sum_{(i,j)\in F} \lambda_{ij} \nu_{ji}(X)
\\ & = \tr(s\nu(X));
\end{align*}
that is, $\tr\left(s\frac{d\nu}{d\nu_\rho}\right)$ is $\nu_\rho$-integrable.
In particular, let $s\in S(\mathcal H)$. By the spectral decomposition theorem, there is a sequence of finite-rank operators $s_n$ converging to $s$ such that $s_n \leq s_{n+1}$. Thus, by the monotone convergence theorem, since $\tr\left(s_n\frac{d\nu}{d\nu_\rho}\right) \leq \tr\left(s_{n+1}\frac{d\nu}{d\nu_\rho}\right)$,
\begin{align*}
\int_X \tr\left(s\frac{d\nu}{d\nu_\rho}\right) d\nu_\rho & = \int_X \lim_{n\rightarrow \infty} \tr\left(s_n\frac{d\nu}{d\nu_\rho}\right) d\nu_\rho
\\ & = \lim_{n\rightarrow \infty} \int_X \tr\left(s_n\frac{d\nu}{d\nu_\rho}\right) d\nu_\rho
\\ & = \lim_{n\rightarrow \infty} \tr(s_n\nu(X))
\\ & = \tr(s\nu(X)).
\end{align*}
Therefore, $x\mapsto I_\mathcal H$ is $\nu$-integrable.

As for uniqueness of the Radon-Nikod\'ym derivative, suppose $\{f_n\}$ is another orthonormal basis for $\mathcal H$ and let $U$ be the unitary in $\mathcal{B}(\mathcal H)$ such that $Ue_k = f_k$. For every $k,l\geq 1$ denote $\nu_{kl}^f$ the complex measure $\nu_{kl}^f(E) = \langle \nu(E)f_l, f_k\rangle, E\in \mathcal O(X).$
We calculate
\begin{align*}
\int_E \left\langle \frac{d\nu}{d\nu_\rho} f_l, f_k\right\rangle d\nu_\rho & = \int_E \left\langle \frac{d\nu}{d\nu_\rho} Ue_l, Ue_k\right\rangle d\nu_\rho
\\ & = \int_E \tr\left(Ue_{kl}U^* \frac{d\nu}{d\nu_\rho}\right) d\nu_\rho
\\ & = \tr\left(Ue_{kl}U^* \int_E \frac{d\nu}{d\nu_\rho} d\nu_\rho\right)
\\ & = \langle \nu(E)Ue_l, Ue_k \rangle
\\ & = \langle \nu(E)f_l, f_k\rangle
\\ & = \nu_{kl}^f(E).
\end{align*}
Hence, 
\[
\left\langle \frac{d\nu}{d\nu_\rho} f_l, f_k\right\rangle = \frac{d\nu_{kl}^f}{d\nu_\rho}
\]
almost always with respect to $\nu_\rho$ by the uniqueness of the Radon-Nikod\'ym derivative.
\end{proof}

\begin{corollary}\label{cor:measandbndd}
Every essentially bounded quantum random variable $f: X \rightarrow \mathcal{B}(\mathcal H)$ is $\nu$-integrable for a POVM $\nu : \mathcal O(X) \rightarrow \mathcal{B}(\mathcal H)$ such that $\frac{d\nu}{d\nu_\rho}$ exists.
\end{corollary}
\begin{proof}
Let $M>0$ such that $\|f(x)\| \leq M$ almost everywhere with respect to $\nu$. 
Thus, $({\rm Re} f)_+, ({\rm Re} f)_-, ({\rm Im} f)_+$ and $({\rm Im} f)_-$ are also  essentially bounded by $M$. By the previous two propositions, since each of these is less than or equal to the $\nu$-integrable function $x\mapsto MI_\mathcal H$ almost everywhere with respect to $\nu$,  they are all $\nu$-integrable.
\end{proof}

In general, for non-positive, not essentially bounded, quantum random variables, it is unlikely that the $\nu_\rho$-integrability of $f_s, s\in \mathcal{T}(\mathcal H)$, would imply the $\nu$-integrability of $f$ and so we need the stronger definition above. One should note that our definition of $\nu$-integrability is stronger than that found in \cite{FPS}.

\begin{example}
Define a POVM $\nu : [0,2] \rightarrow M_2$ by 
\[
\nu(E) = \left[\begin{array}{cc} \mu(E \cap [0,1]) & 0 \\ 0 & \mu(E \cap [1,2])  \end{array}\right]
\]
where $\mu$ is the Lebesgue measure. Let $\rho = \left[\begin{smallmatrix}1/2&0\\ 0 & 1/2\end{smallmatrix}\right]$ and so $\nu_\rho(E) = \frac{1}{2}\mu(E \cap [0,1]) + \frac{1}{2}\mu(E \cap [1,2])$.
Hence, 
\[
\frac{d\nu}{d\nu_\rho} = \left[\begin{array}{cc} 2\chi_{[0,1]} & 0 \\ 0 & 2\chi_{(1,2]} \end{array}\right].
\]
Let $f : [0,2] \rightarrow M_2$ be the self-adjoint quantum random variable defined by 
\begin{align*}
f(x) = \frac{\chi_{(0,1](x)}}{x}\left[\begin{array}{cc} 0 & 1 \\ 1 & 0 \end{array}\right] & = \frac{\chi_{(0,1]}(x)}{2x}\left[\begin{array}{cc} 1&1 \\ 1&1\end{array}\right] - \frac{\chi_{(0,1]}(x)}{2x}\left[\begin{array}{cc} 1&-1\\ -1&1\end{array}\right]
\\ & = f_+(x) - f_-(x).
\end{align*}
Now $f_s$ is $\nu_\rho$-integrable for every $s\in M_2$ because
\[
\left(\frac{d\nu}{d\nu_\rho}(x)\right)^{1/2}f(x)\left(\frac{d\nu}{d\nu_\rho}(x)\right)^{1/2} = 0_2
\]
whereas 
\[
\int_{[0,2]} \tr\left(e_{11}\left(\frac{d\nu}{d\nu_\rho}(x)\right)^{1/2}f_+(x)\left(\frac{d\nu}{d\nu_\rho}(x)\right)^{1/2}\right)d\nu_\rho 
= \int_{[0,2]} \frac{\chi_{(0,1]}(x)}{2x} d\mu = \infty.
\]
Therefore, $f_s$ is $\nu_\rho$-integrable for every state but $f$ is not $\nu$-integrable.
\end{example}

If we consider a quantum random variable $f: X \rightarrow \mathcal{B}(\mathcal H)$, we can define $|f|$ to be the operator analogue of the absolute value function: $|f(x)|:= (f(x)^*f(x))^{1/2}$. 

\begin{proposition}
Let $f: X \rightarrow \mathcal{B}(\mathcal H)$ be a normal quantum random variable and let   $\nu : \mathcal O(X) \rightarrow \mathcal{B}(\mathcal H)$ be a POVM for which $\frac{d\nu}{d\nu_\rho}$ exists. Then $f$ is $\nu$-integrable if and only if $|f|$ is $\nu$-integrable.
\end{proposition}

\begin{proof}
Suppose $A \in \mathcal B(\mathcal H)$ is normal then by functional calculus $({\rm Re} A)_+, ({\rm Re} A)_-, ({\rm Im} A)_+,$ $ ({\rm Im} A)_- \in C^*(I, A)$ and so all commute. Recall that from this one gets that
\[
A^*A = ({\rm Re} A)^2 + ({\rm Im} A)^2 = |{\rm Re} A|^2 + |{\rm Im} A|^2 \geq |{\rm Re} A|^2 \ \textrm{or} \ |{\rm Im} A|^2.
\]
By the operator monotonicity of the square root function we get 
\[
2|A| \geq |{\rm Re} A| + |{\rm Im} A| = ({\rm Re} A)_+ + ({\rm Re} A)_- + ({\rm Im} A)_+ + ({\rm Im} A)_-.
\]
Similarly, 
\begin{align*}
A^*A & = (|{\rm Re} A| + |{\rm Im} A|)^{1/2}(|{\rm Re} A| - |{\rm Im} A|)(|{\rm Re} A| + |{\rm Im} A|)^{1/2}
\\ & \leq (|{\rm Re} A| + |{\rm Im} A|)^2
\end{align*}
which again by operator monotonicity gives that
\[
|A| \leq ({\rm Re} A)_+ + ({\rm Re} A)_- + ({\rm Im} A)_+ + ({\rm Im} A)_-.
\]
This comparability gives by Proposition \ref{prop:fgnu-int} that $f$ is $\nu$-integrable if and only if $|f|$ is $\nu$-integrable.
\end{proof}

In general there exists an operator for which there is no comparability between its absolute value and its four positive summands so integrability does not imply ``absolute'' integrability nor is the converse true. 

In the following theorem, for $i,j\geq 1$, let $s_{ij} = e_{ji}\in \mathcal{B}(\mathcal H)$. Thus, 
\[
\tr(s_{ij}A) = \langle A e_j, e_i\rangle
\]
extracts the $i,j$ entry of $A\in \mathcal{B}(\mathcal H)$. The result can be seen as a generalization of the theory of posterior states in \cite{Ozawa}.

\begin{theorem}\label{thm:nu-integral}
For every POVM $\nu$ for which $\frac{d\nu}{d\nu_\rho}$ exists, the formula 
\[
\int_X f d\nu := \sum_{i,j\geq 1} \left(\int_X f_{s_{ij}} d\nu_\rho\right) \otimes e_{ij}
\]
defines the unique operator on $\mathcal H$ that satisfies
\[
\tr\left(s \int_X f d\nu\right) = \int_X f_s d\nu_\rho
\]
for all $s\in \mathcal{T}(\mathcal H).$ Moreover, $\int_X f d\nu$ is independent of the choice of density operator $\rho$ and orthonormal basis $\{e_i\}$.
\end{theorem}

\begin{proof}
Assume that $f$ is a positive quantum random variable. The general case follows by linearity.

Let $n\geq 1$ and let $s = \sum_{i,j=1}^n c_{ij}s_{ij} = \sum_{i,j=1}^n c_{ij} e_{ji}$ be a finite-rank operator. 
Now,
\begin{align*}
\tr\left( s\int_X f d\nu\right) 
& = \tr\left(\left[\begin{array}{ccc} \sum_{i=1}^n c_{i1}\left( \int_X f_{s_{i1}} d\nu_\rho \right) & * & \\
* & \sum_{i=1}^n c_{i2}\left( \int_X f_{s_{i2}} d\nu_\rho \right)
\\ && \ddots 
\end{array}\right] \right)
\\ 
& = \int_X\left( \sum_{i,j=1}^n c_{ij} f_{s_{ij}}\right) d\nu_\rho 
\\ & = \int_X \left( \sum_{i,j}^n c_{ij}\tr\left( s_{ij}\left(\frac{d\nu}{d\nu_\rho} \right)^{1/2} f \left(\frac{d\nu}{d\nu_\rho} \right)^{1/2} \right) \right)d\nu_\rho 
\\ & = \int_X \left(\tr\left( s\left(\frac{d\nu}{d\nu_\rho} \right)^{1/2} f \left(\frac{d\nu}{d\nu_\rho} \right)^{1/2} \right) \right) d\nu_\rho
\\ & = \int_X f_s d\nu_\rho.
\end{align*}

When $s=\xi\xi^*$, for $\xi\in {\rm span}\{e_1,\dots, e_n\}$, the trace property gives that $f_s$ is a positive function and so
\[
\left\langle \int_X f d\nu \xi, \xi \right\rangle = \tr\left( s\int_X f d\nu\right) = \int_X f_s d\nu_\rho \geq 0.
\]
Thus, $\int_X f d\nu$ is a positive, but possibly unbounded, operator when $f$ is a positive quantum random variable. If $\mathcal H$ is finite-dimensional then we are done.

For every infinite-rank, positive, trace-class operator $s$ there is a sequence of finite-rank, positive operators $s_n$ converging to $s$ in norm such that $s_n \leq s_{n+1}$; this is easily seen from the spectral decomposition of a normal compact operator. Now, by the properties of the trace, we have
\begin{align*}
f_s &= \tr\left( s\left(\frac{d\nu}{d\nu_\rho} \right)^{1/2} f \left(\frac{d\nu}{d\nu_\rho} \right)^{1/2} \right)
\\ &= \tr\left(\left(\left(\frac{d\nu}{d\nu_\rho} \right)^{1/2} f \left(\frac{d\nu}{d\nu_\rho} \right)^{1/2}\right)^{1/2} s\left(\left(\frac{d\nu}{d\nu_\rho} \right)^{1/2} f \left(\frac{d\nu}{d\nu_\rho} \right)^{1/2}\right)^{1/2} \right)
\\ & \geq \tr\left(\left(\left(\frac{d\nu}{d\nu_\rho} \right)^{1/2} f \left(\frac{d\nu}{d\nu_\rho} \right)^{1/2}\right)^{1/2} s_n\left(\left(\frac{d\nu}{d\nu_\rho} \right)^{1/2} f \left(\frac{d\nu}{d\nu_\rho} \right)^{1/2}\right)^{1/2} \right)
\\ &= \tr\left( s_n\left(\frac{d\nu}{d\nu_\rho} \right)^{1/2} f \left(\frac{d\nu}{d\nu_\rho} \right)^{1/2} \right)
\\ & = f_{s_n}.
\end{align*}
Thus, because the sequence $f_{s_n}$ is bounded by the $\nu_\rho$-integrable function $f_s$, by Lebesgue's dominated convergence theorem  we have
\begin{align*}
\int_X f_s d\nu_\rho & = \lim_{n\rightarrow \infty} \int_X f_{s_n}d\nu_\rho
\\ & = \lim_{n\rightarrow \infty} \tr\left( s_n\int_X fd\nu  \right)
\\ & = \tr\left( s\int_X fd\nu  \right),
\end{align*}
where the last equality is a priori true in the extended real numbers as it is the limit of an increasing sequence of positive numbers.

Hence, by linearity, $\tr(s\int_X f d\nu) = \int_X f_s d\nu_\rho$ for all 
$s \in \mathcal T(\mathcal H)$. Therefore, since $\tr\left(s \int_X fd\nu_\rho\right)$ is finite for every state $s\in S(\mathcal H)$ then $\int_X fd\nu_\rho$ must be a bounded operator.

As for uniqueness, suppose $T\in \mathcal{B}(\mathcal H)$ is the bounded operator obtained by integrating and summing over some other orthonormal basis. Thus, the previous argument still holds and
\[
\tr(s T) = \int f_s d\nu_\rho = \tr\left( s \int_X fd\nu\right),\ s\in S(\mathcal H).
\]
The states separate operators on $\mathcal{B}(\mathcal H)$ and so $T = \int_X fd\nu$.

Lastly, suppose $\gamma$ is another full-rank density operator where $\frac{d\nu}{d\nu_\gamma}$ exists. We know that $\nu$, $\nu_\rho$, and $\nu_\gamma$ are all mutually absolutely continuous and so 
\[
\frac{d\nu_\rho}{d\nu_\gamma}\frac{d\nu}{d\nu_\rho} = \sum_{i,j\geq 1} \left(\frac{d\nu_\rho}{d\nu_\gamma}\frac{d\nu_{ij}}{d\nu_\rho}\right) \otimes e_{ij} = \frac{d\nu}{d\nu_\gamma}.
\]
Hence,
\begin{align*}
\left\langle \int_X f d\nu e_j, e_i \right\rangle & = \int_X f_{s_{ij}} d\nu_\rho
\\ & = \int_X \left( \tr\left(s_{ij}\left(\frac{d\nu}{d\nu_\rho} \right)^{1/2} f \left(\frac{d\nu}{d\nu_\rho} \right)^{1/2} \right) \right)\left(\frac{d\nu_\rho}{d\nu_\gamma}\right)d\nu_\gamma
\\ & = \int_X \left(\tr\left(\frac{d\nu_\rho}{d\nu_\gamma}s_{ij}\left(\frac{d\nu}{d\nu_\rho} \right)^{1/2} f \left(\frac{d\nu}{d\nu_\rho} \right)^{1/2} \right) \right)d\nu_\gamma
\\ & = \int_X \left(\tr\left(s_{ij}\left(\frac{d\nu}{d\nu_\gamma} \right)^{1/2} f \left(\frac{d\nu}{d\nu_\gamma} \right)^{1/2} \right) \right)d\nu_\gamma.
\end{align*}
Therefore, the operator $\int_X fd\nu$ is independent of both orthonormal basis and density operator. 
\end{proof}

We extract the last part of the proof as a corollary.

\begin{corollary}\label{cor:exists}
Let $\nu : \mathcal O(X) \rightarrow \mathcal B(\mathcal H)$ be a POVM. If $\frac{d\nu}{d\nu_\rho}$ exists for one full-rank density operator $\nu$ then it exists for any other full-rank density operator $\gamma$. Namely
\[
\frac{d\nu}{d\nu_\gamma} = \frac{d\nu_\rho}{d\nu_\gamma}\frac{d\nu}{d\nu_\rho}.
\]
\end{corollary}

Unsurprisingly, we call the operator $\int_X f d\nu$ the {\em integral} of $f$ with respect to $\nu$. By the previous theorem we see that the integral is linear and takes positive quantum random variables to positive operators.

Analogous with classical measure theory we wish to consider the integral of characteristic functions. It is then necessary to find a suitable definition for a characteristic function. If we take the intuitive approach and define the characteristic functions to be $\chi_E I_{\mathcal H}$, where $\chi_E$ is the classical scalar-valued characteristic function for a measurable set $E$, we'd then expect that $\int_X \chi_E I_{\mathcal H}\, d\nu = \nu(E)$. To test this, consider the inner product:
\begin{align*}
\left\langle \int_X \chi_E I_{\mathcal{H}} \, d\nu\, e_j,e_i\right\rangle & = \tr(s_{i j}\chi_E I_{\mathcal{H}} ) \\
& = \int_X \left(\chi_E I_{\mathcal{H}}\right)_{s_{i j}}d\nu_\rho \\
& = \int_X \tr \left(s_{ij}\left(\frac{d\nu}{d\nu_\rho}\right)^{1/2} \chi_E I_{\mathcal H} \left(\frac{d\nu}{d\nu_\rho}\right)^{1/2}\right)d\nu_\rho \\
& = \int_E \left\langle \frac{d\nu}{d\nu_\rho} e_j, e_i\right\rangle d\nu_\rho = \int \frac{d\nu_{ij}}{d\nu_\rho} d\nu_\rho=\nu_{ij}(E). \\
\end{align*}
This shows that the characteristic functions (as we defined above) satisfy the integral formula; however, they do not span the non-commutative space $L^\infty_{\mathcal H}(X,\nu)$. This indicates that our original definition is needed in order to capture all functions of interest.

One can extend the definition of the integral to the more general setting of an OVM $\nu$ that is in the span of POVMs. Necessarily, $\nu$ needs to be bounded and, as will be seen, not every OVM can be described this way. So suppose $\nu = \nu_1 - \nu_2 + i\nu_3 - i\nu_4$ for POVMs $\nu_1,\nu_2,\nu_3,\nu_4$. Provided that $f: X \rightarrow \mathcal{B}(\mathcal H)$ is $\nu_i$-integrable for $i=1,2,3,4$ then define
\[
\int_X f d\nu = \int_X fd\nu_1 - \int_X fd\nu_2 + i\int_X fd\nu_3 - i\int_X fd\nu_4.
\]

In  \cite[Theorem 3.7]{FPS}, the authors show that for  two POVMs  $\nu,\omega:\mathcal{O}(X)\rightarrow \mathcal{B}(\mathcal H)$, for finite dimensional $\mathcal{H}$,   $\omega \ll_{\rm ac} \nu$ is equivalent to the existence of 
 a bounded quantum random variable $g:X\rightarrow\mathcal{B}(\mathcal H)$, unique up to sets of $\nu$-measure zero, such that
\begin{equation}\label{rn prop}
\omega(E)\,=\,\int_E g\,d\nu ,\;\mbox{ for every }E\in\mathcal{O}(X)\,;
\end{equation}
that is,  $\displaystyle g=\frac{d\omega}{d\nu}$, the Radon-Nikod\'ym derivative of $\omega$ with respect
to $\nu$. 

Unfortunately, for infinite dimensional $\mathcal{H}$, the function $g$ exists only sometimes---in which case it's determined by 
%
\[
\frac{d\omega_{ij}}{d\nu_\rho} = \left\langle \frac{d\omega}{d\nu}\left( \frac{d\nu}{d\nu_\rho}\right)^{1/2}e_j, \left(\frac{d\nu}{d\nu_\rho}\right)^{1/2}e_i \right\rangle = \left( \frac{d\omega}{d\nu}\right)_{s_{ij}}.
\]
A nice characterization of when $g$ exists (and when it does not) does not appear possible. For example:

\begin{example}
Let $\mu$ be Lebesgue measure on $[0,1]$. Let $\mu_1(E) = \mu(E \cap [0,1/2])$ and $\mu_2(E) = \mu(E \cap (1/2, 1])$ and define
\[
\nu = {\rm diag}\left(\mu, \frac{1}{2}\mu_1 + \frac{3}{2}\mu_2, \frac{1}{4}\mu_1 + \frac{7}{4}\mu_2, \dots\right);
\]
that is, $\nu_{ii} = \frac{1}{2^{i-1}}\mu_1 + \frac{2^{i}-1}{2^{i-1}}\mu_2$. So $\nu$ is a POVM with $\nu([0,1]) = I$ and so $\nu$ is  a quantum probability measure.

Let $\rho = {\rm diag}(1/2, 1/4, 1/8, \dots)$ and so
\begin{align*}
\nu_\rho = \sum_{i\geq 1} \frac{1}{2^i}\nu_{ii} & = \sum_{i\geq 1} \frac{1}{2^{2i-1}}\mu_1 + \frac{2^{i}-1}{2^{2i-1}}\mu_2
\\ & = \left(\frac{1}{2}\sum_{i\geq 1} \left(\frac{1}{4}\right)^{i-1}\right)(\mu_1 - \mu_2) + \sum_{i\geq 1} \frac{1}{2^{i-1}}\mu_2
\\ & = \frac{2}{3}\mu_1 + \frac{4}{3}\mu_2.
\end{align*}
Hence,
\[
\left(\frac{d\nu}{d\nu_\rho}\right)_{ii} = \frac{3}{2^i}\chi_{[0,1/2]} + \frac{3(2^i-1)}{2^{i+1}}\chi_{(1/2, 1]}.
\]
Thus, on $[0,1]$, $\frac{d\nu}{d\nu_\rho}$ is injective but not bounded below and so its inverse on its image does not exist.

Now, $\nu_\rho I_\mathcal H$ and $\nu$ are mutually absolutely continuous POVMs into $\mathcal B(\mathcal H)$. We have $\int_E \frac{d\nu}{d\nu_\rho} \nu_\rho I_\mathcal H = \nu(E)$ but if there were a quantum random variable $g: X \rightarrow \mathcal B(\mathcal H)$ such that $\int_E g d\nu = \nu_\rho(E)$ then $g$ restricted to the range of the   $ \frac{d\nu}{d\nu_\rho}$ would be equal to the generalized inverse $ \left(\frac{d\nu}{d\nu_\rho}\right)^{-1}$. Therefore, there is no Radon-Nikod\'ym derivative  $\frac{d\nu_\rho I_\mathcal H}{d\nu}$.
\end{example}

\section{Operator valued decomposition theorems}\label{sec:decomp}

\subsection{Structure of the positive operator-valued measures}\label{sec:POVM}

We state the classical result of Naimark here for completeness. 

\begin{theorem}[Naimark's dilation theorem \cite{Neumark}]
Let $\nu:\mathcal O(X) \rightarrow \mathcal B(\mathcal H)$ be a regular POVM. There exists a Hilbert space $\mathcal K$, a regular, projection-valued measure $\omega : \mathcal O(X) \rightarrow \mathcal B(\mathcal K)$ and a bounded operator $V : \mathcal K \rightarrow \mathcal H$ such that
\[
\nu(E) = V\omega(E)V^*, \ \ \ E\in \mathcal O(X).
\]
\end{theorem}

Despite not involving a dilation, the following theorem is reminiscent of Naimark's dilation theorem. 

\begin{theorem}\label{thm:douglas}
Let $\nu : \mathcal O(X) \rightarrow \mathcal B(\mathcal H)$ be a POVM. There exists a quantum probability measure $\omega : \mathcal O(X) \rightarrow \mathcal B(\mathcal H)$ such that
\[
\nu(E) = \nu(X)^{1/2}\omega(E)\nu(X)^{1/2}, \ \ \ E\in\mathcal O(X).
\]
Moreover, $\omega$ is unique on the range of $\nu(X)$.
\end{theorem}
\begin{proof}
 Note that $\nu(X)$ is not necessarily invertible (otherwise, we could trivially define $\omega(E)= \nu(X)^{-1/2}\nu(E)\nu(X)^{-1/2}$). For every $E\in \mathcal O(X)$ we have that $\nu(E) \leq \nu(X)$. By Douglas' Lemma \cite{DouglasLem}, there exists a unique $C_E\in \mathcal B(\mathcal H)$ such that
\begin{align*}
\bullet \ \   & \nu(E)^{1/2} = \nu(X)^{1/2}C_E, 
\\ \bullet  \  \ &  \ker C_E = \ker \nu(E)^{1/2} = \ker \nu(E) 
\\ \bullet  \ \ &  {\rm ran}\ C_E \subseteq \overline{{\rm ran}\ \nu(X)^{1/2}} = \overline{{\rm ran}\ \nu(X)}
\end{align*}
By similar reasoning to Douglas' proof, $C_EC_E^*$ is uniquely defined since $\nu(X)$ is bijective on ${\rm ran}\ \nu(X)$.

Thus, define $\omega_1(E) = C_EC_E^*, E\in \mathcal O(X)$ to get that
\[
\nu(X)^{1/2}\omega_1(E)\nu(X)^{1/2} = \nu(X)^{1/2}C_EC_E^*\nu(X)^{1/2} = \nu(E).
\]
It is immediate that $\omega_1(E)$ is positive for all $E\in \mathcal O(X)$ and $\omega_1(X) = P_{{\rm ran}\ \nu(X)}$, the projection onto the range space of $\nu(X)$.
Lastly, suppose $\{E_i\}_{i \in I}$ is a finite or countable set of disjoint measurable subsets. 
We know that 
\begin{align*}
\sum_{i\in I} \nu(X)^{1/2}\omega_1(E_i)\nu(X)^{1/2} &= \sum_{i\in I} \nu(E_i) 
\\ & = \nu(\cup_{i\in I} E_i) 
\\ & = \nu(X)^{1/2}\omega_1(\cup_{i\in I} E_i)\nu(X)^{1/2}
\end{align*}
where the sum converges in the ultraweak topology. Hence,
\[
\sum_{i\in I} \omega_1(E_i) = \omega_1(\cup_{i\in I} E_i)
\]
where the sum converges in the ultraweak topology by the uniqueness of the operator $C_EC_E^*.$ Therefore, $\omega_1$ is a POVM.

Let $\mu$ be any probability measure on $X$. Define for $E\in \mathcal O(X)$
\[
\omega_2(E) = \mu(E)(I_\mathcal H - \omega_1(X)) = \mu(E)P_{\ker \nu(X)}.
\]
Therefore, $\omega = \omega_1 + \omega_2$ is a quantum probability measure such that 
\[
\nu(E) = \nu(X)^{1/2}\omega(E)\nu(X)^{1/2}, \ \ \ E\in \mathcal O(X).
\]
\end{proof}

Recall that a C$^*$-convex combination of functions $f_i : Y \rightarrow \mathcal B(\mathcal H), 1\leq i\leq k$ is 
\[
\sum_{i=1}^k A_i^*fA_i \ \ \ \textrm{where} \ \ \ \sum_{i=1}^n A_i^*A_i = I_\mathcal H.
\]

\begin{corollary}\label{cor:cstarconvex}
Let $\nu : \mathcal O(X) \rightarrow \mathcal B(\mathcal H)$ be a quantum probability measure such that $\nu = \nu_1 + \cdots + \nu_n$ for POVMs $\nu_i : \mathcal O(X) \rightarrow \mathcal B(\mathcal H)$. There exists quantum probability measures $\gamma_i : \mathcal O(X) \rightarrow \mathcal B(\mathcal H)$ such that
\[
\nu(E) = \sum_{i=1}^k \nu_i(X)^{1/2}\gamma_i(E)\nu_i(X)^{1/2}, \ \ E\in \mathcal O(X), 
\]
that is, every decomposition of a quantum probability measure can be realized  as a C$^*$-convex combination of quantum probability measures.
\end{corollary}

\subsection{Hahn-Jordan decomposition}\label{sec:HJ}

\begin{theorem}[Hahn Decomposition]\cite[Theorem  29.A]{Halmos}\label{thm:ClassHahn}
Let $\mu$ be a signed measure on $\Sigma$. Then there exists disjoint sets $S,T\in\Sigma$ such that $S\cup T=X$,  $\mu(A)\geq 0$ for all $A\in \Sigma$ with $A\subseteq S$, and $\mu(B)\leq 0$ for all $B\in \Sigma$ with $B\subseteq T$.
\end{theorem}

For any two measures $\mu$ and $\lambda$ defined on $\Sigma$, we say $\mu$ and $\lambda$ are \emph{singular} (denoted $\mu \perp \lambda$) if there exists disjoint sets $A,B\in\Sigma$, where $A\cup B=X$, such that $\mu$ is zero for all measurable subsets of B and $\lambda$ is zero for all measurable subsets of A. Note that this relation is clearly symmetric (i.e. if $\mu\perp\lambda$ then $\lambda\perp\mu$).

\begin{theorem}[Jordan Decomposition]\cite[Theorem 19.6]{Nielsen}\label{thm:ClassJordan}
Let $\mu$ be a signed measure on $\Sigma$. Then there exists unique positive measures $\mu^+$ and $\mu^-$ such that $\mu=\mu^+-\mu^-$, with the property that at least one of $\mu^+$ and $\mu^-$ is finite and $\mu^+ \perp \mu^-$.
\end{theorem}

Theorems \ref{thm:ClassHahn} and \ref{thm:ClassJordan} are closely related in that the Jordan decomposition of $\mu$ follows as a consequence of the Hahn decomposition of $X$. To see this, consider decomposing $X$ into two sets $S$ and $T$ according to Theorem \ref{thm:ClassHahn}. We can then construct the measures $\mu_1(E)=\mu(E\cap S)$ and $\mu_2(E)=-\mu(E\cap T)$. These measures end up satisfying all the properties described in Theorem \ref{thm:ClassJordan}, and so, by uniqueness, we have $\mu^+(E)=\mu_1(E)$ and $\mu^-(E)=\mu_2(E)$; thereby allowing for a constructive proof of Theorem \ref{thm:ClassJordan}. See 
\cite[Section 29]{Halmos} for more details. 
An interesting consequence of the uniqueness of the Jordan decomposition (and the non-uniqueness of the Hahn decomposition) means that if $S'$ and $T'$ are another Hahn decomposition of $X$, then we must have $\mu(E\cap S)=\mu(E\cap S')$ and $\mu(E\cap T)=\mu(E\cap T')$ for all $E\in\Sigma$.

The measures $\mu^+$ and $\mu^-$ from Theorem \ref{thm:ClassJordan} are given the names \emph{upper variation} and \emph{lower variation} of $\mu$, respectively. These measures are then used to define the measure $|\mu|$, called the \emph{total variation} of $\mu$, where $|\mu|(E)=\mu^+(E)+\mu^-(E)$ and is defined for all $E\in\Sigma$.

It is an immediate observation that there can be no Hahn decomposition beyond the dimension 1 case since a self-adjoint operator in general is not either positive or negative but a mixture of the two. However, one can still show that under some conditions an OVM is the linear combination of four POVMs. This is what the higher-dimensional analogue of the Hahn-Jordan theorem becomes.

Every OVM $\nu$ is easily seen to be the linear combination of two self-adjoint OVMs 
\[
\nu(E) =  \frac{\nu(E) + \nu(E)^*}{2} + i\left(\frac{\nu(E) - \nu(E)^*}{2i}\right), \ E\in \mathcal O(X).
\]
One would perhaps like to attempt to decompose a self-adjoint, bounded OVM $\nu: \mathcal O(X) \rightarrow \mathcal{B}(\mathcal H)$ as $\nu = \nu_+ - \nu_-$ where these POVMs are defined by $\nu_+(E) = \nu(E)_+$ and $\nu_-(E) = \nu(E)_-$ for all $E\in \mathcal O(X)$; that is, just decompose each self-adjoint operator $\nu(E)$ into its positive and negative parts. This would have the advantage of recovering something akin to the Hahn decomposition since $\nu_+(E)\nu_-(E) = 0$, a form of singularity. However, this approach proves to be too na\"{i}ve. 

Hadwin \cite{Hadwin} shows that there is a bijective correspondence between regular, bounded OVMs $\nu : \mathcal O(X) \rightarrow \mathcal{B}(\mathcal H)$ and bounded linear maps $\phi_\nu : C(X) \rightarrow \mathcal{B}(\mathcal H)$ by way of the equations
\[
\langle \phi_\nu(f)x,y\rangle = \int f d\nu_{x,y}, \ \ \forall x,y\in \mathcal H, \forall f\in C(X),
\]
where $\nu_{x,y}(E) = \langle \nu(E)x,y\rangle$ for $x,y\in \mathcal H$ is a complex regular measure.

A regular, bounded OVM $\nu$ will be called {\em completely bounded} if $\phi_\nu$ is completely bounded.

\begin{theorem}[Hadwin, {\cite[Theorem 20]{Hadwin}}]
Let $\nu: \mathcal O(X) \rightarrow \mathcal{B}(\mathcal H)$ be a regular, bounded OVM. Then $\nu$ is the linear combination of four regular POVMs if and only if $\nu$ is completely bounded.
\end{theorem}

The proof is nicely outlined in Chapter 8 of \cite{Paulsen} and uses Wittstock's decomposition theorem for completely bounded maps and Stinespring's theorem that every positive linear map $C(X) \rightarrow \mathcal{B}(\mathcal H)$ is completely positive and thus completely bounded.

A major consequence of this theorem is that there are OVMs which cannot be written as linear combinations of POVMs. Hadwin additionally gives an example of a completely bounded OVM which does not have finite total variation. Paulsen \cite[Chapter 8]{Paulsen} shows that this gives an example of a completely bounded, regular, self-adjoint OVM $\nu$ where $\nu_+$ and $\nu_-$ do not define POVMs but by the previous theorem this $\nu$ still can be written as the linear combination of two POVMs.

\subsection{Lebesgue decomposition}
In classical measure theory, we have the following result. 
\begin{theorem}[Lebesgue Decomposition]\cite[Theorem 15.14]{Nielsen}\label{thm:ClassLebesgue} Let $\mu$ and $\lambda$ be measures on $\Sigma$ with $\lambda$ being $\sigma$-finite. Then there exists unique measures $\lambda_a$ and $\lambda_s$ such that $\lambda=\lambda_a  + \lambda_s$, $\lambda_a \ll_{ac} \mu$ and $\lambda_s \perp \mu$.
\end{theorem}

A positive  operator-valued version of the Lebesgue decomposition theorem can be stated as follows. The proof is identical to that of the first part of the proof in \cite[Theorem 15.14]{Nielsen}; we do not need the whole proof as a POVM is the equivalent of a finite measure. 

\begin{theorem}\label{thm:OpLebesgue} Let $\nu:  \mathcal{O}(X)\rightarrow B(\mathcal{H}_1)$ and $\omega:  \mathcal{O}(X)\rightarrow B(\mathcal{H}_2)$ be POVMs. Then there exists unique POVMs   $\omega_a$ and $\omega_s$ such that $\omega=\omega_a  + \omega_s$, $\omega_a \ll_{ac} \nu$ and $\omega_s \perp \nu$.
\end{theorem}

\subsection{Atomic and nonatomic decomposition}

An \emph{atom} for a positive measure $\mu$ defined on $\Sigma$ is a set $A$ of non-zero measure, such that for each subset $B\subseteq A$ either $\mu(B)=0$ or $\mu(B)=\mu(A)$. If every set of non-zero measure contains an atom then $\mu$ is \emph{atomic}. On the other hand, if $\mu$ has no atoms then $\mu$ is \emph{non-atomic}. Analogous definitions can be made with respect to a POVM $\nu$.


\begin{theorem}[]\cite{Johnson70}\label{thm:ClassAtom} Let $\mu$ be a positive $\sigma$-finite measure on $\Sigma$. Then there exists positive measures $\mu_a$ and $\mu_{na}$ such that $\mu=\mu_{a} + \mu_{na}$, where $\mu_{a}$ is atomic and $\mu_{na}$ is non-atomic. Additionally, $\mu_a$ and $\mu_{na}$ may be chosen such that $\mu_a\mathcal \perp \mu_{na}$, which, under these conditions, makes them unique.
\end{theorem}

The following result states that every POVM can be written uniquely as the sum of an atomic POVM and a non-atomic POVM. The proof is not all that different from the classical setting \cite{Johnson70}.

\begin{theorem}\label{thm:a+na}
Every POVM can be written uniquely as the sum of an atomic POVM and a non-atomic POVM.
\end{theorem}


The classical decomposition makes use of the notion of a positive measure  being $\mathcal{S}$-singular with respect to another positive measure \cite{Johnson70}; this weaker notion is equivalent to the notion of singular because we are dealing with $\sigma$-finite measures \cite[Theorem 3.3]{Johnson67}. 
 
\begin{proof}
Let $\nu:\mathcal{O}(X)\rightarrow \mathcal{B}(\mathcal{H})$ be a POVM. We wish to show there exists quantum measures $\nu_1, \nu_2$ such that $\nu_1$ is atomic, $\nu_2$ is non-atomic, and $\nu=\nu_1+\nu_2$. 

There is at most a countable set of disjoint atoms $\{A_n\}$ because $\nu$ is mutually absolutely continuous to a finite classical measure $\nu_\rho$.  Suppose this family is exhaustive meaning that for any atom $A$ we have that $\nu(A \cap (\cup A_n)^C) = 0$.
Let $X_a = \cup A_n$ and $X_{na} = X \setminus X_a$, then $X_a, X_{na} \in \mathcal{O}(X)$. We will see that these are the atomic and non-atomic supports for $\nu$. For each $E\in \mathcal{O}(X)$, define 
\begin{eqnarray*}
\nu_1(E)&=& \nu(E \cap X_a) \\ 
\nu_2(E)&=& \nu(E \cap X_{na}). 
\end{eqnarray*}
Since $\nu_1$ and $\nu_2$ are both defined by restrictions to measurable sets they are both POVMs automatically and 
\[
\nu(E) = \nu(E\cap X_a) + \nu(E \cap X_a^C) = \nu_1(E) + \nu_2(E).
\]
Thus, $\nu = \nu_1 + \nu_2$.

Now let us establish that $\nu_1$ is atomic. Suppose $E\in \mathcal{O}(X)$ with $\nu_1(E) \neq 0$. This implies that
\[
0\neq \nu_1(E) = \nu(E \cap X_a) = \nu(\cup (E \cap A_n))
\]
and by the disjointness of the sets we must have an $n\in \mathbb N$ such that $\nu(E\cap A_n)\neq 0$. Thus, $E\cap A_n$ is an atom since it is a subset of an atom that has nonzero measure. Hence, $\nu_1$ is atomic.

Next suppose that $E\in \mathcal{O}(X)$ is an atom of $\nu_2$. So,
\[
\nu_2(E\cap X_{na}) = \nu_2(E)
\]
which implies that $E\cap X_{na}$ is also an atom of $\nu_2$ and thus an atom of $\nu$. However, by the way we defined $\nu_1$ for any atom we have $\nu(E\cap X_{na}) = \nu_1(E\cap X_{na}) = \nu(E \cap X_{na} \cap X_a) = 0$, a contradiction. Therefore, $\nu_2$ has no atoms and is thus non-atomic.

Uniqueness follows in the exact same manner as for the classical proof. 
\end{proof}

\begin{theorem}
Let $\nu : \mathcal O(X) \rightarrow \mathcal{B}(\mathcal H)$ be a quantum probability measure. There exists atomic and non-atomic quantum probability measures $\nu_a$ and $\nu_{na}$, respectively, and $P\in \mathcal{B}(\mathcal H)$, $0\leq P\leq I$ such that
\[
\nu(E) = P^{1/2}\nu_a(E)P^{1/2} + (I-P)^{1/2}\nu_{na}(E)(I-P)^{1/2}, \ E\in \mathcal O(X).
\]
\end{theorem}
\begin{proof}
By the previous theorem we have $\nu$ decomposed into its atomic and non-atomic parts, $\nu = \nu_a + \nu_{na}$.
By Corollary \ref{cor:cstarconvex} there exists quantum probability measures $\gamma_a, \gamma_{na} : \mathcal O(X) \rightarrow \mathcal B(\mathcal H)$ such that 
\[
\nu(E) = \nu_a(X)^{1/2}\gamma_a(E)\nu_a(X)^{1/2} + \nu_{na}(X)^{1/2}\gamma_{na}(E)\nu_{na}(X)^{1/2}
\]
for every $E\in \mathcal O(X)$.
A look back at the proof of Theorem \ref{thm:douglas} gives that in the case of $\gamma_a = (\gamma_a)_1 + (\gamma_a)_2$ we already have that $(\gamma_a)_1$ is atomic and $(\gamma_a)_2$ can be chosen to be atomic, for $\mu$ a Dirac mass say. Thus, $\gamma_a$ is an atomic quantum probability measure. In the same way, $\gamma_{na}$ can also be chosen to be non-atomic. 

Therefore, the conclusion follows for $P = \nu_a(X) = I - \nu_{na}(X)$.
\end{proof}

The classical version of the following result can be found in \cite[Theorem 2.4]{Johnson70}; the  proof is analogous to  that found in \cite{Johnson70} other than we again note that the concept of $\mathcal{S}$-singular is identical to singularity in our context. 

\begin{proposition} \label{prop:abscontna}
Suppose $\nu_i: \mathcal{O}(X)\rightarrow \mathcal{B}(\mathcal H_i)$, $i=1,2$ are POVMs such that $\nu_1\leac \nu_2$. If   $\nu_1$ is non-atomic then $\nu_2$ is non-atomic. If   $\nu_1$ is  atomic then $\nu_2$ is atomic. Hence, if $\nu_1$ is nonzero and atomic, then $\nu_2$ has an atom. 
\end{proposition}

\begin{corollary}
Suppose $\nu_i : \mathcal{O}(X)\rightarrow \mathcal{B}(\mathcal H_i)$, $i=1,2$ are POVMs such that $\nu_1 \leac \nu_2$. If $\nu_i = (\nu_i)_a + (\nu_i)_{na}$ is the unique decomposition of $\nu_i$ into its atomic and nonatomic parts then $(\nu_1)_a \leac (\nu_2)_a$ and $(\nu_1)_{na} \leac (\nu_2)_{na}.$
\end{corollary}
\begin{proof}
Let $E\in \mathcal O(X)$ be such that $\nu_2|_E = (\nu_2)_a|_{E}$ and $\nu_2|_{X \setminus E} = (\nu_2)_{na}|_{X\setminus E}$, that is $E$ is the support of the atomic part of $\nu_2$. Now
\[
\nu_1|_{X\setminus E} \leac \nu_2|_{X\setminus E} = (\nu_2)_{na}|_{X\setminus E}\ \  \ \ \textrm{and} \ \ \ \ \nu_1|_{E} \leac \nu_2|_{E} = (\nu_2)_{a}|_{E}
\]
which by the previous proposition implies that $\nu_1|_{X\setminus E}$ is nonatomic and $\nu_1|_{E}$ is atomic. By uniqueness of the atomic/non-atomic decomposition the result follows.
\end{proof}


\section{Clean and informationally complete OVMs}\label{sec:cleanIC}

The atomic/nonatomic decomposition leads to some applications in the study of quantum probability measures in quantum information theory. 
\begin{definition} If $\nu$ is a quantum probability measure on $(X,\mathcal{O}(X))$, then
the  \emph{range} of $\nu$ is the set 
\[
\mathcal{R}_\nu\,=\,\{\nu(E)\,:\,E\in\mathcal{O}(X)\}\,\subset\,\mathcal{B}(\mathcal{H})_+,
\]
and 
the \emph{measurement space of $\nu$} is the vector space
\[
\mathcal{T}_\nu\,=\,\left(\mbox{\rm Span}_{\mathbb C}\,\mathcal{R}_\nu\right)^{\sigma{\rm-wk}}  \,\subset\,\mathcal{B}(\mathcal{H})\,,
\]
the ultraweak closure of all linear combinations of operators of the form $\nu(E)$, for $E\in\mathcal{O}(X)$.
\end{definition}

\begin{definition}
A quantum probability measure $\nu$ is \emph{informationally complete} if, for any density operators
$\rho_1,\rho_2\in\mathcal{B}(\mathcal{H})$, $\tr(\rho_1 \nu(E))=\tr(\rho_2 \nu(E))$ for every $E\in\mathcal{O}(X)$ implies $\rho_1=\rho_2$.
\end{definition}

The next result is a consequence of the Hahn-Banach separation theorem:

\begin{proposition}\label{ic-span} The following statements are equivalent for a quantum probability measure $\nu$:
\begin{enumerate}
\item $\nu$ is informationally complete;
\item $\mathcal{T}_\nu=\mathcal{B}(\mathcal{H})$.
\end{enumerate}
\end{proposition}
\begin{proof}
 A quantum probability measure $\nu$ is informationally complete if and only if $\{\nu(E) : E\in \mathcal O(X)\}$ separates the state space. This is the same as $\{\nu(E) : E\in \mathcal O(X)\}$ separating the trace-class operators $\mathcal{T}(\mathcal H) = \mathcal{B}(\mathcal H)_*$, which in turn is equivalent to $\mathcal T_\nu = \overline{{\rm span}}\{\nu(E) : E\in\mathcal O(X)\}$ separating $\mathcal{B}(\mathcal H)_*.$ By a standard corollary to the Hahn-Banach separation theorem there are no strict subspaces of the dual, here $\mathcal{B}(\mathcal H)$, that separate a Banach space, here $\mathcal{B}(\mathcal H)_*$. The result follows. 
\end{proof}

The following definition can be found in \cite{FFP}. 

 \begin{definition} A \emph{measurement basis} for a quantum probability measure $\nu$ is a finite or countably infinite set $\mathcal B_\nu$ of positive operators 
such that
\begin{enumerate}
\item[{\rm (i)}] $\mathcal B_\nu=\{\nu(E)\,:\,E\in\mathcal F_\nu\}$ for some finite or countable family $\mathcal F_\nu\subset\mathcal{O}(X)$ of pairwise disjoint sets, 
\item[{\rm (ii)}] for every $Z\in \mathcal{T}_\nu$ there exists a unique sequence $\{\alpha_{A,Z}\}_{A\in \mathcal{B}_\nu}$ of complex numbers such that $Z=\sum_{A\in \mathcal B_\nu }\alpha_{A,Z}A$ in the weak*-topology,
\item[{\rm (iii)}]  for every $A\in  \mathcal B_\nu$, the coefficient functional $\varphi_A(Z)=\alpha_{A,Z}$, $Z\in \mathcal{T}_\nu$, is a normal positive linear functional.
\end{enumerate}
If $E_0=X\setminus\left(\bigcup_{E\in\mathcal F_\nu}E\right)$, then the operator $A_0=\nu(E_0)$
is called the \emph{basis residual} for $\mathcal B_\nu$;  if $A_0=0$, then  $\mathcal B_\nu$ is said to admit a \emph{trivial basis residual}.
\end{definition}
Note that a measurement basis is a particular construction that does not necessarily exist for a given  quantum probability measure.


\begin{proposition}\cite{FFP}\label{structure} If $\{A_1, A_2, \dots \}$ is a measurement basis for a quantum probability measure $\nu$, 
then there exist finite positive measures
$\mu_j:\mathcal{O}(X)\rightarrow \mathcal{B}(\mathcal{H})$, $j\geq 1$, such that each  $\mu_j\leac\nu$ 
 and
\[
\nu(E)\,=\,\sum_{j\geq 1} \mu_j(E)A_j\,,\;\mbox{for all }E\in\mathcal{O}(X)\,,
\]
where convergence of the sum is with respect to the ultraweak topology.
\end{proposition}
Note that in \cite{FFP} this proposition states that the $\mu_i$ are only signed measures but in fact condition (iii) of the definition of measurement basis ensures that they are positive.


The following result gives a method for creating    informationally complete quantum probability measures.

\begin{proposition}\label{prop:converseIC}
Let $\mu_j:\mathcal{O}(X)\rightarrow \mathcal{B}(\mathcal{H})$, $j\geq 1$, be mutually singular probability measures   and let $\{A_j \}$ be linearly independent, positive, span $\mathcal{B}(\mathcal H)$, and $\sum_{j\geq 1} A_{j}= I$.   Define
\[
\nu(E)\,=\,\sum_{j\geq 1}\mu_{j}(E)A_{j}\,,\;\mbox{for all }E\in\mathcal O(X)\,.
\]
Then $\nu$ is informationally complete. 
\end{proposition}

\begin{proof}
By the hypothesis that  $A_1, A_2, \dots$ span $\mathcal B(\mathcal H)$, $\nu$ will automatically be informationally complete by Proposition \ref{ic-span}.  

\end{proof}

Proposition \ref{prop:converseIC} allows us to create examples of informationally complete atomic quantum probability measures by choosing $\mu_{j}$ to be Dirac point masses where each point is an isolated point, or non-atomic quantum probability measures by choosing $\mu_{j}$ non-atomic.
Thus we have examples of atomic and non-atomic informationally complete quantum probability measures, and in light of Proposition \ref{ic-span}, we have the following corollary. 

\begin{corollary}
The property of atomic/non-atomic does not show up in the measurement space. That is, an operator system $\mathcal T$ could be a measurement space for both an atomic quantum probability measure and a non-atomic quantum probability measure. 
\end{corollary} 

The following definitions regarding the clean order come from quantum information theory literature, where a \emph{quantum channel} $\Phi:\mathcal T(\mathcal H)\rightarrow\mathcal T(\mathcal K)$ is a completely positive, trace preserving, linear map between the trace-class operators acting on   Hilbert spaces $\mathcal{H}$ and $\mathcal K$, respectively. The dual map $\Phi^*: \mathcal T(\mathcal K)^*\rightarrow\mathcal T(\mathcal H)^*$ (i.e.\  $\Phi^*: \mathcal B(\mathcal K)\rightarrow\mathcal B(\mathcal H)$) is then a completely positive, unital linear map. 
\begin{definition}\label{clean order} {\rm (\cite{pellonpaa2011})}
Let $\nu_i:\mathcal{O}(X)\rightarrow \mathcal{B}(\mathcal{H}_i)$, $i=1,2$, be quantum probability measures.
\begin{enumerate}
\item $\nu_1$ is \emph{cleaner than} $\nu_2$, denoted by
$\nu_2 \ll_{\rm cl} \nu_1$, if $\nu_2=\Phi^*\circ\nu_1$ for some quantum
channel $\Phi:\mathcal T(\mathcal H_2)\rightarrow\mathcal T(\mathcal H_1)$.
\item $\nu_1$ and $\nu_2$ are \emph{cleanly equivalent}, denoted by $\nu_2 \ce \nu_1$, if
$\nu_1 \ll_{\rm cl} \nu_2$ and $\nu_2 \ll_{\rm cl} \nu_1$.
\item $\nu_1$ is \emph{clean} if $\nu_2\ll_{\rm cl}\nu_1$ for every quantum probability measure $\nu_2$ satisfying $\nu_1\ll_{\rm cl}\nu_2$.
\end{enumerate}
\end{definition}

\begin{proposition}\label{prop:clean-na}Let $\nu_i:\mathcal{O}(X)\rightarrow \mathcal{B}(\mathcal{H}_i)$, $i=1,2$, be cleanly equivalent quantum probability measures. If $\nu_1$ is non-atomic, then $\nu_2$ is non-atomic.
\end{proposition}

\begin{proof}
Suppose $\nu_1,\nu_2$ are cleanly equivalent. Then $\nu_2=\phi^*\circ\nu_1$ and  $\nu_1=\psi^*\circ\nu_2$ for some quantum
channels  $\phi:\mathcal{T}(\mathcal{H}_2)\rightarrow \mathcal{T}(\mathcal{H}_1)$ and  $\psi:\mathcal{T}(\mathcal{H}_1)\rightarrow \mathcal{T}(\mathcal{H}_2)$. Let $E\in\mathcal O(X)$ and assume $\nu_2(E)\neq0$. We wish to show that $\nu_2(F)\neq 0$ and $\nu_2(F)\neq \nu_2(E)$ for some subset $F \subsetneq E$. Note that $\nu_2(E)\neq0$ implies $\phi^*\circ\nu_1(E)\neq0$, yielding $\nu_1(E)\neq 0$ by linearity of $\phi^*$. Since $\nu_1$ is non-atomic, we have $\nu_1(F)\neq 0$ for some $F\subset E$, implying that $\nu_1(F)=\psi^*\circ\nu_2(F)\neq 0$, yielding $\nu_2(F)\neq 0$ by linearity of $\psi^*$. 

\medskip

It remains to show that $\nu_2(F)\neq \nu_2(E)$. To this end, we consider $\nu_1(F)\neq \nu_1(E)$ (since $\nu_1$ is non-atomic); this is equivalent to $\psi^*\circ\nu_2(F)\neq \psi^*\circ\nu_2(E)$. Applying the channel $\phi^*$ to both sides of the inequality, and noting that $\phi^*\circ\psi^*|_{\mathcal R_{\nu_2}}= I_{\mathcal{H}}$, we obtain  $\nu_2(F)\neq \nu_2(E)$ as desired.  
\end{proof}

\section*{Acknowledgements}
 S.P.\ was supported by NSERC Discovery Grant number 1174582, the Canada Foundation for Innovation, and the Canada Research Chairs Program. S.P.\ thanks Doug Farenick for helpful discussions at the initial stage of this work.

\end{document}